\documentclass[11pt]{article}
\title{{\bf A sufficient condition for the existence of an anti-directed 2-factor in a directed graph}}
\author{Ajit A. Diwan\\
Department of Computer Science and Engineering\\
Indian Institute of Technology\\
Mumbai, India\\
\\
Josh B. Frye\\
Department of Mathematics\\
Illinois State University \\
Normal, IL 61790-4520\\
Email: jbfrye@ilstu.edu\\
\\
Michael J. Plantholt \\
Department of Mathematics\\
Illinois State University \\
Normal, IL 61790-4520\\
E-mail: mikep@ilstu.edu\\
\\
Shailesh K. Tipnis\\
Department of Mathematics\\
Illinois State University \\
Normal, IL 61790-4520\\
Email: tipnis@ilstu.edu\\
}

\begin{document}
\newtheorem{theorem}{Theorem}
\newtheorem{lemma}{Lemma}
\newtheorem{corollary}{Corollary}
\newenvironment{proof}{{\bf Proof}.}{\hspace{3mm}\rule{3mm}{3mm}}
\newenvironment{mainproof}{{\bf Proof of Theorem 7}.}{\hspace{3mm}\rule{3mm}{3mm}}
\newenvironment{cor1proof}{{\bf Proof of Corollary 1}.}{\hspace{3mm}\rule{3mm}{3mm}}
\newenvironment{cor2proof}{{\bf Proof of Corollary 2}.}{\hspace{3mm}\rule{3mm}{3mm}}
\newenvironment{cor3proof}{{\bf Proof of Corollary 3}.}{\hspace{3mm}\rule{3mm}{3mm}}

\newenvironment{subproof}{{\bf Proof of Claim}.}{\hspace{3mm}\rule{3mm}{3mm}}
\maketitle\thispagestyle{empty}
\begin{abstract}
Let $D$ be a directed graph with vertex set $V$ and order $n$. An {\em anti-directed
(hamiltonian) cycle} $H$ in $D$ is a (hamiltonian) cycle in the
graph underlying $D$ such that no pair of consecutive arcs in $H$
  form a directed path in
$D$.  An {\em anti-directed 2-factor} in $D$ is a vertex-disjoint collection of anti-directed cycles in $D$ that span $V$.
It was proved in \cite{BJMPT}  that if  the indegree and the outdegree of each vertex of
$D$ is greater than $\frac{9}{16}n $ then $D$ contains an anti-directed hamiltonian cycle. In this paper we prove that given a directed
graph $D$, the problem of determining whether $D$ has an anti-directed 2-factor is NP-complete, and we use a proof technique similar
to the one used in \cite{BJMPT} to prove that if  the indegree and the outdegree of each vertex of
$D$ is greater than $\frac{25}{48}n $ then $D$ contains an anti-directed 2-factor.
\end{abstract}

\section{Introduction}
Let $G$ be a multigraph with vertex set $V(G)$ and edge set $E(G)$. For a vertex $v \in V(G)$, the degree
of $v$ in $G$, denoted by ${\rm deg}(v,G)$ is the number of edges of $G$ incident on $v$.
Let $\delta(G) = {\rm min}_{v \in V(G)}\{{\rm deg}(v,G)\}$.
The simple graph underlying $G$ denoted by simp($G$) is the graph obtained from $G$ by replacing
all multiple edges by single edges. A {\em 2-factor} in $G$ is a collection of vertex-disjoint cycles
that span $V(G)$.
Let $D$ be a directed graph with vertex set $V(D)$ and arc set
$A(D)$. For a vertex $v \in V(D)$, the {\em outdegree} (respectively, {\em indegree}) of $v$ in $D$ denoted
by $d^{+}(v,D)$ (respectively, $d^{-}(v,D)$)
is the number of arcs of $D$ directed out of $v$ (respectively, directed into $v$). Let $\delta(D) = {\rm min}_{v \in V(D)}\{{\rm min}\{d^{+}(v,D),
d^{-}(v,D)\}\}$. The {\em multigraph underlying $D$} is the multigraph obtained from $D$ by ignoring the directions
of the arcs of $D$. A {\em directed (Hamilton) cycle} $C$ in $D$ is a (Hamilton) cycle in
the multigraph underlying $D$ such that all pairs of consecutive arcs in $C$
 form a directed path in
$D$.
 An {\em anti-directed (Hamilton) cycle} $C$ in $D$ is a (Hamilton) cycle in the
multigraph underlying $D$ such that no pair of consecutive arcs in $C$
 form a directed path in
$D$. A {\em directed 2-factor} in $D$ is a collection of vertex-disjoint directed cycles in $D$ that
span $V(D)$. An {\em anti-directed 2-factor} in $D$ is a collection of vertex-disjoint anti-directed cycles in $D$ that
span $V(D)$.
Note that every anti-directed cycle in $D$ must have an even number of vertices.
 We refer the reader to ([1,7]) for
all terminology and notation that is not defined in this paper.

The following classical theorems by Dirac \cite{Dirac} and Ghouila-Houri \cite{GH}
give sufficient conditions for the existence of a Hamilton cycle in a graph $G$ and for
the existence of a directed Hamilton cycle in a directed graph $D$ respectively.
\begin{theorem}{\rm \cite{Dirac}} If $G$ is a graph of order $n \geq 3$ and
$\delta(G) \geq \frac{n}{2}$, then $G$ contains a Hamilton cycle.
\end{theorem}

\begin{theorem}{\rm \cite{GH}} If $D$ is a directed graph of order $n$ and
$\delta(D) \geq \frac{n}{2}$, then $D$ contains a directed
Hamilton cycle.
\end{theorem}
Note that if $D$ is a directed graph of even order $n$ and
$\delta(D) \geq \frac{3}{4}n$ then $D$ contains an anti-directed
Hamilton cycle. To see this, let $G$ be the multigraph underlying $D$ and let $G'$ be the
subgraph of $G$ consisting of the parallel edges of $G$. Now, $\delta(D) \geq \frac{3}{4}n$ implies that
$\delta({\rm simp}(G')) \geq \frac{n}{2}$ and hence Theorem 1 implies that simp($G'$) contains a Hamilton
cycle which in turn implies that $D$ contains an anti-directed Hamilton cycle.

The following theorem by Grant \cite{Grant} gives a sufficient condition for the existence
of an anti-directed Hamilton cycle in a directed graph $D$.
\begin{theorem} {\rm \cite{Grant}} If $D$ is a directed graph with even order $n$ and if
$\delta(D) \geq \frac{2}{3}n + \sqrt{n{\rm log}(n)}$ then $D$ contains an anti-directed
Hamilton cycle.
\end{theorem}
In his paper Grant \cite{Grant} conjectured that the theorem above can be strengthened to
assert that if $D$ is a directed graph with even order $n$ and if
$\delta(D) \geq \frac{1}{2}n$ then $D$ contains an anti-directed
Hamilton cycle. Mao-cheng Cai \cite{Mao} gave a counter-example to this conjecture.
\noindent
In \cite{BJMPT} the following sufficient condition for the existence of an anti-directed Hamilton
cycle in a directed graph was proved.
\begin{theorem} {\rm \cite{BJMPT}} Let $D$ be a directed graph of even order $n$ and suppose that $\frac{1}{2} < p < \frac{3}{4}$.
If $\delta(D) \geq pn$ and $n > \frac{{\rm ln}(4)}{\left(p - \frac{1}{2}\right){\rm ln}\left(\frac{p + \frac{1}{2}}{\frac{3}{2} - p}\right)}$, then $D$ contains an anti-directed
Hamilton cycle.
\end{theorem}
\noindent
It was shown in \cite{BJMPT} that Theorem 4 implies the following corollary that is an improvement on the result
in Theorem 3.
\noindent
\begin{corollary} {\rm \cite{BJMPT}}
If $D$ is a directed graph of even order $n$ and $\delta(D) > \frac{9}{16}n $ then $D$ contains an anti-directed
Hamilton cycle.
\end{corollary}
\noindent
The following theorem (see \cite{digraphsbook}) gives a necessary and sufficient condition for
the existence of a directed 2-factor in a digraph $D$.
\begin{theorem}
A directed graph $D = (V,A)$ has a directed 2-factor if and only if $|\bigcup_{v \in X}N^{+}(v)| \geq |X|$ for all $X \subseteq V$.
\end{theorem}
\noindent
We note here that given a directed graph $D$ the problem of determining whether $D$ has a directed Hamilton cycle is known to
be NP-complete, whereas, there exists an O$(\sqrt{n}m)$ algorithm (see \cite{digraphsbook}) to check if
a directed graph $D$ of order $n$ and size $m$ has a directed 2-factor. On the other hand, the following theorem proves that given
a directed graph $D$, the problem of determining whether $D$ has a directed 2-factor is NP-complete. We are indebted to Sundar Vishwanath for
pointing out the short proof of Theorem 6 given below.
\noindent
\begin{theorem} {\rm \cite{Sundar}} Given a directed graph $D$, the problem of determining whether $D$ has an anti-directed 2-factor.
is NP-complete.
\end{theorem}
\begin{proof}
Clearly the the problem of determining whether $D$ has an anti-directed 2-factor is in NP.
A graph $G$ is said to be $k$-edge colorable if the edges of $G$ can be colored with $k$ colors in such a way that
no two adjacent edges receive the same color. It is well known that given a cubic graph $G$, it is NP-complete to
determine if $G$ is 3-edge colorable. Now, given a cubic graph $G = (V,E)$, construct a directed graph $D = (V,A)$,
where for each $\{u,v\}$ $\in$ $E$, we have the oppositely directed arcs $(u,v)$ and $(v,u)$ in $A$.
It is clear that $G$ is 3-edge colorable if and only if $D$ contains an anti-directed 2-factor. This proves that the
the problem of determining whether a directed graph $D$ has an anti-directed 2-factor
is NP-complete.
\end{proof}
\\
\\
\noindent
In Section 1 of this paper we prove the following theorem that gives a sufficient condition for the existence of an anti-directed 2-factor
in a directed graph.
\begin{theorem} Let $D$ be a directed graph of even order $n$ and suppose that $\frac{1}{2} < p < \frac{3}{4}$.
If $\delta(D) \geq pn$ and $n > \frac{{\rm ln}(4)}{\left(p - \frac{1}{2}\right){\rm ln}\left(\frac{p + \frac{1}{2}}{\frac{3}{2} - p}\right)} (??)$, then $D$ contains an anti-directed
2-factor.
\end{theorem}
\noindent
In Section 1 we will show that Theorem 7 implies the following corollary.
\noindent
\begin{corollary} {\rm \cite{BJMPT}}
If $D$ is a directed graph of even order $n$ and $\delta(D) > \frac{25}{48}n $ then $D$ contains an anti-directed
2-factor.
\end{corollary}

\section{Proof of Theorem 7 and its Corollary}
A partition of a set $S$ with $|S|$ being
even into
$S = X \cup Y$ is an {\em equipartition} of $S$ if $|X| = |Y| = \frac{|S|}{2}$.
The proof of Theorem 4 mentioned in the introduction made extensive use of the following
theorem by Chv\'atal \cite{Chvatal}.
\begin{theorem}{\rm \cite{Chvatal}} Let $G$ be a bipartite graph of even order $n$ and with equipartition $V(G) = X \cup Y$.
Let $(d_{1},d_{2},\ldots,d_{n})$ be the degree sequence of $G$ with
$d_{1} \leq d_{2}\leq \ldots \leq d_{n}$. If $G$ does not contain a Hamilton cycle, then
for some $i \leq \frac{n}{4}$ we have that $d_{i} \leq i$ and $d_{\frac{n}{2}} \leq \frac{n}{2} - i$.
\end{theorem}
We prepare for the proof of Theorem 7 by proving Theorems 10 and 11 which give necessary degree conditions
(similar to those in Theorem 8) for the non-existence of a 2-factor in a bipartite graph $G$ of even order $n$ with equipartition $V(G) = X \cup Y$.
\\
\noindent
Let $G = (V,E)$ be a bipartite graph of even order $n$ and with equipartition $V(G) = X \cup Y$. For $U \subseteq X$ (respectively
$U \subseteq Y$) define $N^{(2)}(U)$ as being the multiset of vertices $v \in Y$ (respectively $v \in X$) such that $(u,v) \in E$ for
some $u \in U$ and with $v$ appearing twice in  $N^{(2)}(U)$ if there are two or more vertices  $u \in U$ with $(u,v) \in E$
and $v$ appearing once in  $N^{(2)}(U)$ if there is exactly one  $u \in U$ with $(u,v) \in E$.
We will use the following theorem by Ore \cite{Ore} that gives a necessary and sufficient condition for the non-existence of
a 2-factor in a bipartite graph of even order $n$ with equipartition $V(G) = X \cup Y$.
\begin{theorem} Let $G = (V,E)$ be a bipartite graph of even order $n$ and with equipartition $V(G) = X \cup Y$.
$G$ contains no 2-factor if and only if there exists some $U \subseteq X$ such that $|N^{(2)}(U)| < 2|U|$.
\end{theorem}
For a bipartite graph $G = (V,E)$ of even order $n$ and with equipartition $V(G) = X \cup Y$,
a set $U \subseteq X$ or $U \subseteq Y$ is defined to be a {\em deficient} set of vertices
in $G$ if $|N^{(2)}(U)| < 2|U|$.
\\
\\
\noindent
We now prove four Lemmas that will be used in the proof of Theorems 10 and 11.
\begin{lemma} Let $G$ be a bipartite graph of even order $n$ and with equipartition $V(G) = X \cup Y$.
If $U$ is a minimal deficient set of vertices in $G$ then $2|U| - 2 \leq |N^{(2)}(U)|$.
\end{lemma}
\begin{proof}
Clear by the minimality of $U$.
\end{proof}
\begin{lemma}
Let $G$ be a bipartite graph of even order $n$ and with equipartition $V(G) = X \cup Y$, and let
$U$ be a minimal deficient set of vertices in $G$. Let $M \subseteq N(U)$ be the set of vertices in $N(U)$
that are adjacent to exactly one vertex in $U$. Then, no vertex of $U$ is adjacent to more than one vertex of $M$.
\end{lemma}
\begin{proof}
If a vertex $u \in U$ is adjacent to two vertices of $M$, since $U$ is a deficient set of vertices in $G$,
we have that
$|N^{(2)}(U - u)| \leq |N^{(2)}(U)| - 2 < 2|U| - 2 = 2|U - u|$. This implies that $U - u$ is a deficient set
of vertices in $G$, which in turn contradicts the minimality of $U$.
\end{proof}
\begin{lemma}
Let $G$ be a bipartite graph of even order $n$ and with equipartition $V(G) = X \cup Y$, and suppose that $G$ does not contain
a 2-factor. If $U$ is a minimal deficient set in $G$ with $|U| = k$, then ${\rm deg}(u) \leq k$ for each $u \in U$ and
$|\{u \in U: {\rm deg}(u) \leq k - 1\}| \geq k -1$.
\end{lemma}
\begin{proof} Suppose that ${\rm deg}(u) \geq k + 1$ for some $u \in U$ and let
 $M \subseteq N(U)$ be the set of vertices in $N(U)$
that are adjacent to exactly one vertex in $U$. Then Lemma 2 implies that $u$ is adjacent to at most one vertex
in $M$ which implies that $u$ is adjacent to at least $k$ vertices in $N(U) - M$. This implies that
$|N^{(2)}(U)| \geq 2k$, which contradicts the assumption that $U$ is a deficient set. This proves that
${\rm deg}(u) \leq k$ for each $u \in U$. If two vertices in $U$ have degree $k$ then similarly Lemma 2 implies
that $|N^{(2)}(U)| \geq 2k$, which contradicts the assumption that $U$ is a deficient set. This proves the second
part of the Lemma.
\end{proof}
\begin{lemma}
Let $G = (V,E)$ be a bipartite graph of even order $n$ and with equipartition $V(G) = X \cup Y$ and suppose that $U \subseteq X$
is a minimal deficient set in $G$. Let $Y_{0} = \{v \in Y: v \not\in N(U)\}$, $Y_{1} = \{v \in Y: |U \cap N(v)| = 1\}$,
and $Y_{2} = \{v \in Y: |U \cap N(v)| \geq 2\}$. Let $U^{*} = Y_{0} \cup Y_{1}$. Then $U^{*}$ is a deficient set in $G$.
\end{lemma}
\begin{proof}
Let $X_{0} = X - U, X_{1} = \{u \in U: (u,v) \in E \ {\rm for\  some\ } v \in Y_{1}\}$, and $X_{2} = U - X_{1}$.
Note that $|X| = |Y|$ implies that $|X_{0}| + |X_{1}| + |X_{2}| = |Y_{0}| + |Y_{1}| + |Y_{2}|$. Now, since by Lemma 2 we have that $|X_{1}| = |Y_{1}|$,
this implies that $|X_{0}| + |X_{2}| = |Y_{0}| + |Y_{2}|$. Since $U$ is a deficient set we have that $|N^{(2)}(U)| =
|Y_{0}| + 2|Y_{2}| < 2|U| = 2(|X_{1}| + |X_{2}|$. Hence, $|Y_{1}| + 2(|X_{0}| + |X_{2}| - |Y_{0}|) < 2(|X_{1}| + |X_{2}|)$, which in turn implies
that $2|X_{0}| + |X_{1}| < 2(|Y_{0}| + |Y_{1}|)$. This proves that $U^{*}$ is a deficient set in $G$.
\end{proof}
\\
\\
\noindent
We are now ready to prove two theorems which give necessary degree conditions
(similar to those in Theorem 8) for the non-existence of a 2-factor in a bipartite graph $G$ of even order $n$ with equipartition $V(G) = X \cup Y$.
\begin{theorem}
Let $G$ be a bipartite graph of even order $n = 4s \geq 12$ and with equipartition $V(G) = X \cup Y$.
Let $(d_{1},d_{2},\ldots,d_{n})$ be the degree sequence of $G$ with
$d_{1} \leq d_{2}\leq \ldots \leq d_{n}$. If $G$ does not contain a 2-factor, then either
\begin{itemize}
\item[(1)] for some $k \leq \frac{n}{4}$ we have that $d_{k} \leq k$ and $d_{k-1} \leq k - 1$, or,
\item[(2)] $d_{\frac{n}{4} - 1} \leq \frac{n}{4} - 1$.
\end{itemize}
\end{theorem}
\begin{proof}
We will prove that for some $k \leq \frac{n}{4}$, $G$ contains $k$ vertices with degree at most $k$,
and that of these $k$ vertices, $(k-1)$ vertices have degree at most $(k-1)$, or, that $G$ contains at least
$\frac{n}{4} - 1$ vertices of degree at most $\frac{n}{4} - 1$.
\\
\noindent
Since $G$ does not contain a 2-factor, Theorem 9 implies that $G$ contains a deficient set of vertices. Let
$U \subseteq X$ be a minimal deficient set of vertices in $G$. If $|U| \leq \frac{n}{4}$, then Lemma 3 implies
that statement (1) is true and the result holds.
\\
\noindent
Now suppose that $|U| > \frac{n}{4}$. As in the statement of Lemma 4, let
 $Y_{0} = \{v \in Y: v \not\in N(U)\}$, $Y_{1} = \{v \in Y: |U \cap N(v)| = 1\}$,
and $Y_{2} = \{v \in Y: |U \cap N(v)| \geq 2\}$. Let $U^{*} = Y_{0} \cup Y_{1}$. Then Lemma 4 implies that
$U^{*}$ is a deficient set in $G$. If $|U^{*}| \leq \frac{n}{4}$ then again statement (1) is true and the result holds.
\\
\noindent
Now suppose that $|U^{*}| > \frac{n}{4}$,
and as in the proof of Lemma 4, let
$X_{0} = X - U, X_{1} = \{u \in U: (u,v) \in E \ {\rm for\  some\ } v \in Y_{1}\}$, and $X_{2} = U - X_{1}$.
We have that ${\rm deg}(u) \leq 1 + |Y_{2}|$ for each $u \in U$, and hence we may assume that $|Y_{2}| \geq \frac{n}{4} - 1$, else the result
holds. Similarly, since ${\rm deg}(u) \leq 1 + |X_{0}|$ for each $u \in U^{*}$,
we may assume that $|X_{0}| \geq \frac{n}{4} - 1$. Note that $|U| > \frac{n}{4}$ and
$|X_{0}| \geq \frac{n}{4} - 1$ implies that $|U| = \frac{n}{4} + 1$, and that
$|U^{*}| > \frac{n}{4}$ and
$|Y_{2}| \geq \frac{n}{4} - 1$ implies that $|U^{*}| = \frac{n}{4} + 1$. Now, since $U$ is a minimal deficient set of vertices
in $G$, Lemma 1 implies that $|X_{1}| = 2$ or $X_{1} = 3$. If $|X_{1}| = 2$ then at least $\frac{n}{4} - 1$ of the vertices in $U$
must have degree at most $\frac{n}{4} - 1$, and statement (2) of the theorem is true.
Finally, if $|X_{1}| = 3$ then at least $\frac{n}{2} - 4$ (and hence at least $\frac{n}{4} - 1$ because $n \geq 12$) of the vertices in each of $U$
and $U^{*}$ 
must have degree at most $\frac{n}{4} - 1$, and statement (2) of the theorem is true.
\end{proof}
\begin{theorem}
Let $G$ be a bipartite graph of even order $n = 4s + 2\geq 14$ and with equipartition $V(G) = X \cup Y$.
Let $(d_{1},d_{2},\ldots,d_{n})$ be the degree sequence of $G$ with
$d_{1} \leq d_{2}\leq \ldots \leq d_{n}$. If $G$ does not contain a 2-factor, then either
\begin{itemize}
\item[(1)] for some $k \leq \frac{(n - 2)}{4}$ we have that $d_{k} \leq k$ and $d_{k-1} \leq k - 1$, or,
\item[(2)] $d_\frac{(n-2)}{2} \leq \frac{(n-2)}{4}$.
\end{itemize}
\end{theorem}
\begin{proof}
We will prove that for some $k \leq \frac{n}{4}$, $G$ contains $k$ vertices with degree at most $k$,
and that of these $k$ vertices, $(k-1)$ vertices have degree at most $(k-1)$, or, that $G$ contains at least
$\frac{(n-2)}{2}$ vertices of degree at most $\frac{(n-2)}{4}$.
\\
\noindent
Since $G$ does not contain a 2-factor, Theorem 9 implies that $G$ contains a deficient set of vertices. Without loss of
generality let
$U \subseteq X$ be a minimum cardinality deficient set of vertices in $G$. If $|U| \leq \frac{(n - 2)}{4}$, then Lemma 3 implies
that statement (1) is true and the result holds.
\\
\noindent
Now suppose that $|U| > \frac{(n - 2)}{4}$. As in the statement of Lemma 4, let
 $Y_{0} = \{v \in Y: v \not\in N(U)\}$, $Y_{1} = \{v \in Y: |U \cap N(v)| = 1\}$,
and $Y_{2} = \{v \in Y: |U \cap N(v)| \geq 2\}$. Let $U^{*} = Y_{0} \cup Y_{1}$. Then Lemma 4 implies that
$U^{*}$ is a deficient set in $G$. Since $U$ is a minimum cardinality deficient set of vertices in $G$,
 we have that$|U^{*}| \geq |U| > \frac{(n - 2)}{4}$.
\\
\noindent
Now, as in the proof of Lemma 4, let
$X_{0} = X - U, X_{1} = \{u \in U: (u,v) \in E \ {\rm for\  some\ } v \in Y_{1}\}$, and $X_{2} = U - X_{1}$.
We have that ${\rm deg}(u) \leq 1 + |Y_{2}|$ for each $u \in U$, and hence we may assume that $|Y_{2}| \geq \frac{(n - 2)}{4} - 1$, else the result
holds. Similarly, since ${\rm deg}(u) \leq 1 + |X_{0}|$ for each $u \in U^{*}$,
we may assume that $|X_{0}| \geq \frac{(n - 2)}{4} - 1$.
 Note that $|U| > \frac{(n - 2)}{4}$ and
$|X_{0}| \geq \frac{(n - 2)}{4} - 1$ implies that $\frac{(n - 2)}{4} + 1 \leq |U| \leq \frac{(n - 2)}{4} + 2$.
We now examine the two cases: $|U| = \frac{(n - 2)}{4} + 1$ and $|U| = \frac{(n - 2)}{4} + 2$.
\begin{itemize}
\item[(1)] $|U| = \frac{(n - 2)}{4} + 1$. In this case we must have that $|X_{0}| = \frac{(n - 2)}{4}$.
Note that $|X_{1}| \leq 3$ because if $|X_{1}| \geq 4$ then since $U$ is a minimal deficient set of vertices,
we would have that $|Y_{2}| \leq  \frac{(n - 2)}{4} - 2$, a contradiction to the assumption at this point that
$|Y_{2}| \geq  \frac{(n - 2)}{4} - 1$. We now examine the following four subcases separately.
\begin{itemize}
\item[(1)a] $|X_{1}| = 0$. In this case we have that $|Y_{1}| = 0$ and $|X_{2}| = \frac{(n - 2)}{4} + 1$.
Since $U$ is a minimal deficient set of vertices, Lemma 1 implies that $|Y_{2}| = \frac{(n - 2)}{4}$ and $|Y_{0}| = \frac{(n - 2)}{4} + 1$.
Thus, $X_{2} \cup Y_{0}$ is a set of $\frac{n}{2} + 1$ vertices of degree at most $\frac{(n - 2)}{4}$ which meets the requirement
of the theorem..
\item[(1)b] $|X_{1}| = 1$. In this case we have that $|Y_{1}| = 1$ and $|X_{2}| = \frac{(n - 2)}{4}$.
Since $U$ is a minimal deficient set of vertices, Lemma 1 implies that $|Y_{2}| = \frac{(n - 2)}{4}$ and $|Y_{0}| = \frac{(n - 2)}{4}$.
Thus, $X_{2} \cup Y_{0}$ is a set of $\frac{n}{2} + 1$ vertices of degree at most $\frac{(n - 2)}{4}$ each as required by the theorem.
\item[(1)c] $|X_{1}| = 2$. In this case we have that $|Y_{1}| = 2$ and $|X_{2}| = \frac{(n - 2)}{4} - 1$.
Since $U$ is a minimal deficient set of vertices, Lemma 1 implies that $|Y_{2}| = \frac{(n - 2)}{4} - 1$ and $|Y_{0}| = \frac{(n - 2)}{4}$.
Thus, $X_{2} \cup X_{1} \cup Y_{0}$ is a set of $\frac{n}{2}$ vertices of degree at most $\frac{(n - 2)}{4}$ which meets the requirement
of the theorem.
\item[(1)d] $|X_{1}| = 3$. In this case we have that $|Y_{1}| = 3$ and $|X_{2}| = \frac{(n - 2)}{4} - 2$.
Since $U$ is a minimal deficient set of vertices, Lemma 1 implies that $|Y_{2}| = \frac{(n - 2)}{4} - 1$ and $|Y_{0}| = \frac{(n - 2)}{4} - 1$.
Thus, $X_{2} \cup X_{1} \cup Y_{0}$ is a set of $\frac{n}{2} - 1$ vertices of degree at most $\frac{(n - 2)}{4}$ as required by the theorem.
\end{itemize}
\item[(2)] $|U| = \frac{(n - 2)}{4} + 2$. In this case we have that $|X_{0}| = \frac{(n - 2)}{4} - 1$. Since $U$ is a minimum cardinality
deficient set of vertices, we also have that $|U^{*}| = |U| = \frac{(n - 2)}{4} + 2$. Hence we now have that $|Y_{2}| = |X_{0}| =
\frac{(n - 2)}{4} - 1$. Thus, $U \cup U^{*}$ is a set of $\frac{n}{2} + 3$ vertices of degree at most $\frac{(n - 2)}{4}$ which meets the requirement
of the theorem.
\end{itemize}
\end{proof}

\begin{lemma} Let $x, y, r$ be positive numbers such that $x \geq y$ and $r < y$. Then
$\frac{(x+r)(x-r)}{(y+r)(y-r)} \geq {(\frac{x}{y})}^{2}$.
\end{lemma}
\begin{proof}
$y^{2}(x^{2} - r^{2}) \geq (y^{2} - r^{2})x^{2}$, so the result follows.
\end{proof}
\\
\\
\noindent
\begin{mainproof}
For an equipartition of $V(D)$ into $V(D) = X \cup Y$,
let $B(X \rightarrow Y)$ be the bipartite directed graph with vertex set $V(D)$, equipartition
$V(D) = X \cup Y$, and with $(x,y) \in A(B(X \rightarrow Y))$ if and only if $x \in X$, $y \in Y$, and, $(x,y) \in A(D)$.
Let $B(X,Y)$ denote the bipartite graph underlying $B(X \rightarrow Y)$. It is clear that
$B(X,Y)$ contains a Hamilton cycle if and only if $B(X \rightarrow Y)$ contains an anti-directed Hamilton cycle.
We will prove that there exists an equipartition of $V(D)$ into $V(D) = X \cup Y$ such that $B(X,Y)$ contains a Hamilton
cycle.

In the argument below, we make the simplifying assumption that $d^{+}(v) = d^{-}(v) = \delta(D)$ for each $v \in V(D)$.
It is straightforward (see the remark at the end of the proof) to see that the argument extends to the case in which some indegrees or outdegrees are greater than $\delta(D)$.
\\
\noindent
 Let $v \in V(D)$. Let $n_{k}$ denote the number of equipartitions of $V(D)$ into $V(D) = X \cup Y$
for which ${\rm deg}(v,B(X,Y)) = k$. Since $v \in X$ or $v \in Y$ and since $d^{+}(v) = d^{-}(v) = \delta(D)$, we have that
$n_{k} = 2{\delta \choose k}{n - \delta - 1 \choose \frac{n}{2} - k}$. Note that if $k > \frac{n}{2}$ or
if $k < \delta - \frac{n}{2} +1$ then $n_{k} = 0$.
Thus the total number of equipartitions of $V(D)$ into $V(D) = X \cup Y$ is
\begin{equation}
T = \sum_{k = \delta - \frac{n}{2} +1}^{\frac{n}{2}} n_{k} = \sum_{k = \delta - \frac{n}{2} +1}^{\frac{n}{2}}2{\delta \choose k}{n - \delta - 1 \choose \frac{n}{2} - k} = {n \choose \frac{n}{2}}.
\end{equation}
Denote by $N = {n \choose \frac{n}{2}}$ the total number of equipartitions of $V(D)$. For a particular equipartition of $V(D)$ into $V(D) = X_{i} \cup Y_{i}$, let
$(d_{1}^{(i)},d_{2}^{(i)},\ldots,d_{n}^{(i)})$ be the degree sequence of $B(X_{i},Y_{i})$ with
$d_{1}^{(i)} \leq d_{2}^{(i)}\leq \ldots \leq d_{n}^{(i)}$, $i = 1,2,\ldots, N$, and, let $P_{i} = \{j: d_{j}^{i} \leq \frac{n}{4}\}$.
If $B(X_{i},Y_{i})$ does not contain a Hamilton
cycle then Theorem 8 implies that there exists $k \leq \frac{n}{4}$ such that $d_{k}^{i} \leq k$ and hence,
$|\{d_{j}^{i} : d_{j}^{i} \leq k, j = 1,2,\ldots,n\}| \geq k$. This in turn implies that
$\sum_{j\in P_{i}} \frac{1}{d_{j}^{i}} \geq 1$.
Hence, the number of equipartitions of $V(D)$ into $V(D) = X \cup Y$ for which
$B(X,Y)$ does not contain a Hamilton cycle is at most
\begin{equation}
S = n\left(\frac{n_{2}}{2} + \frac{n_{3}}{3} + \ldots + \frac{n_{\lfloor \frac{n}{4}\rfloor}}{\lfloor\frac{n}{4}\rfloor}\right)
\end{equation}
Thus, to show that there exists an equipartition of $V(D)$ into $V(D) = X \cup Y$ such that $B(X,Y)$ contains a Hamilton
cycle, it suffices to show that $T > S$, i.e.,
\begin{equation}
\sum_{k = \delta - \frac{n}{2} +1}^{\frac{n}{2}}2{\delta \choose k}{n - \delta - 1 \choose \frac{n}{2} - k}
> n \sum_{k = 2}^{\lfloor \frac{n}{4}\rfloor} \frac{2{\delta \choose k}{n - \delta - 1 \choose \frac{n}{2} - k}}{k}
\end{equation}
We break the proof of (3) into three cases.
\\
\noindent
{\bf Case 1}: $n = 4m$ and $\delta = 2d$ for some positive integers $m$ and $d$.
\newline
For $i = 0,1,\ldots,\frac{n}{4} - 2$, let $A_{i} = n_{(d + i)} = 2{\delta \choose d + i}{n - \delta - 1 \choose 2m - d - i}$,
and let $B_{i} = n_{(\frac{n}{4}-i)} = 2{\delta \choose m - i}{n - \delta - 1 \choose m + i}$.
Clearly, (3) is satisfied if we can show that
\begin{equation}
A_{i} > \frac{nB_{i}}{\frac{n}{4} - i},\  {\rm for\  each}\  i = 0,1,\ldots,\frac{n}{4} - 2.
\end{equation}
We prove (4) by recursion on $i$.
We first show that $A_{0} > \frac{nB_{0}}{\frac{n}{4}}$, i.e. $n_{\frac{\delta}{2}} > n\left(\frac{n_{\frac{n}{4}}}{\frac{n}{4}}\right)
= 4n_{\frac{n}{4}}$. Let $\delta = \frac{n}{2} + s$. We have that
\begin{eqnarray}
\frac{A_{0}}{B_{0}}& = &\frac{(\frac{n}{4})!(\delta - \frac{n}{4})!(\frac{n}{4})!(\frac{3n}{4} - \delta - 1)!}
{\frac{\delta}{2}!\frac{\delta}{2}!(\frac{n}{2}-\frac{\delta}{2})!(\frac{n}{2}-\frac{\delta}{2}- 1)!} \nonumber \\
&=& \frac{(\frac{n}{4})!(\frac{n}{4} + s)!(\frac{n}{4})!(\frac{n}{4} - s -1)!}{(\frac{n}{4} + \frac{s}{2})!(\frac{n}{4} + \frac{s}{2})!
(\frac{n}{4} - \frac{s}{2})!
(\frac{n}{4} - \frac{s}{2} - 1)!} \nonumber \\
&=&\frac{(\frac{n}{4} + s)(\frac{n}{4} + s - 1)\ldots(\frac{n}{4} + \frac{s}{2} + 1)(\frac{n}{4})(\frac{n}{4} - 1)\ldots(\frac{n}{4} - \frac{s}{2} + 1)}{(\frac{n}{4} + 1)(\frac{n}{4} + 2)\ldots (\frac{n}{4} + \frac{s}{2})(\frac{n}{4} - \frac{s}{2} - 1)(\frac{n}{4} - \frac{s}{2} - 2)\ldots(\frac{n}{4} - s)}\nonumber
\end{eqnarray}
\noindent
Now, applications of Lemma 1 give
\begin{eqnarray}
\frac{A_{0}}{B_{0}}& \geq &\frac{{(\frac{n}{4} + \frac{3s}{4} + \frac{1}{2})}^{\frac{s}{2}}}
{{(\frac{n}{4} + \frac{s}{4} + \frac{1}{2})}^{\frac{s}{2}}}
\frac{{(\frac{n}{4} - \frac{s}{4} + \frac{1}{2})}^{\frac{s}{2}}}
{{(\frac{n}{4} - \frac{3s}{4} - \frac{1}{2})}^{\frac{s}{2}}}\nonumber \\
& \geq & \frac{{(\frac{n}{4} + \frac{s}{4} + \frac{1}{2})}^{s}}
{{(\frac{n}{4} - \frac{s}{4})}^{s}}
\end{eqnarray}
Since $\delta \geq pn$, we have that $s = \delta - \frac{n}{2} \geq (p - \frac{1}{2})n$.
Thus, (5) gives
\begin{equation}
\frac{A_{0}}{B_{0}} \geq {\left(\frac{\frac{n}{4} + \frac{(p - \frac{1}{2})n}{4}}{\frac{n}{4} -\frac{(p - \frac{1}{2})n}{4} }\right)}
^{\left(p - \frac{1}{2}\right)n}= {\left(\frac{p + \frac{1}{2}}{\frac{3}{2} - p}\right)}^{\left(p - \frac{1}{2}\right)n}
\end{equation}
Because $ n > \frac{{\rm ln}(4)}{\left(p - \frac{1}{2}\right){\rm ln}\left(\frac{p + \frac{1}{2}}{\frac{3}{2} - p}\right)}$, (6) implies that
$\frac{A_{0}}{B_{0}} > 4$, thus proving (4) for $i = 0$.
\\
We now turn to the recursive step in proving (4) and assume that
$A_{k} > \frac{nB_{k}}{\frac{n}{4} - k},\  {\rm for\  }\  0 < k < \frac{n}{4} - 2$.
We will show that
\begin{equation}
\frac{A_{k + 1}}{A_{k}} \geq \left(\frac{\frac{n}{4} - k}{\frac{n}{4} - k -1}\right) \frac{B_{k + 1}}{B_{k}}
\end{equation}
This will suffice because (7) together with the recursive hypothesis implies that
$A_{k+1} \geq \left(\frac{\frac{n}{4} - k}{\frac{n}{4} - k -1}\right) \frac{A_{k}}{B_{k}}B_{k+1}
> \left(\frac{\frac{n}{4} - k}{\frac{n}{4} - k -1}\right) \frac{n}{\frac{n}{4} - k}B_{k+1} = \frac{n}{\frac{n}{4} - k - 1}B_{k+1}$.
We have that
\[ \frac{A_{k+1}}{A_{k}} = \frac{{\delta \choose \frac{\delta}{2} + k + 1}{n - \delta - 1 \choose
\frac{n}{2} - \frac{\delta}{2} - k - 1}}
{{\delta \choose \frac{\delta}{2} + k }{n - \delta - 1 \choose
\frac{n}{2} - \frac{\delta}{2} - k}} = \frac{\left(\frac{\delta}{2} - k\right)\left(\frac{n}{2} - \frac{\delta}{2} - k\right)}
{\left(\frac{\delta}{2} + k + 1\right)\left(\frac{n}{2} - \frac{\delta}{2} + k\right)},\]
\[ {\rm and}, \frac{B_{k+1}}{B_{k}} = \frac{{\delta \choose \frac{n}{4} - k - 1}{n - \delta - 1 \choose
\frac{n}{4} + k + 1}}
{{\delta \choose \frac{n}{4} - k }{n - \delta - 1 \choose
\frac{n}{4} + k}} = \frac{\left(\frac{n}{4} - k\right)\left(\frac{3n}{4} - \delta - k - 1\right)}
{\left(\delta - \frac{n}{4} + k + 1\right)\left(\frac{n}{4} + k + 1\right)}.\]
Hence, letting $\delta = \frac{n}{2} + s$, we have that
\begin{eqnarray}
\frac{\left(\frac{A_{k+1}}{A_{k}}\right)}{\left(\frac{B_{k+1}}{B_{k}}\right)}& = &
\frac{\left(\frac{\delta}{2} - k\right)\left(\frac{n}{2} - \frac{\delta}{2} - k\right)\left(\delta - \frac{n}{4} + k + 1\right)\left(\frac{n}{4} + k + 1\right)}
{\left(\frac{n}{4} - k\right)\left(\frac{3n}{4} - \delta - k - 1\right)\left(\frac{\delta}{2} + k + 1\right)\left(\frac{n}{2} - \frac{\delta}{2} + k\right)}\nonumber \\
& = & \frac{\left(\frac{n}{4} + \frac{s}{2} - k\right)\left(\frac{n}{4} - \frac{s}{2} - k\right)\left(\frac{n}{4} + s + k + 1\right)
\left(\frac{n}{4} + k + 1\right)}
{\left(\frac{n}{4} - k\right)\left(\frac{n}{4} - s - k - 1\right)\left(\frac{n}{4} + \frac{s}{2} + k + 1\right)
\left(\frac{n}{4} - \frac{s}{2} + k \right)}
\end{eqnarray}
Note that in equation (8) we have, $\frac{\left(\frac{n}{4} + \frac{s}{2} - k\right)}{\left(\frac{n}{4} - k\right)} > 1$,
$\frac{\left(\frac{n}{4} + s + k + 1\right)} {\left(\frac{n}{4} + \frac{s}{2} + k + 1\right)} > 1$,
$\frac{\left(\frac{n}{4} + k + 1\right)} {\left(\frac{n}{4} - \frac{s}{2} + k \right)} > 1$,
and in addition because $k < \frac{n}{4}$, it is easy to verify that
$\frac{\left(\frac{n}{4} - \frac{s}{2} - k\right)} {\left(\frac{n}{4} - s - k - 1\right)}
> \frac{\left(\frac{n}{4} - k\right)}
{\left(\frac{n}{4} - k - 1\right)}$. Now (8) implies (7) which in turn proves (4).
This completes the proof of Case 1.
\\
\\
\noindent
{\bf Case 2}: $n = 4m$ and $\delta = 2j + 1$ for some positive integers $m$ and $j$.
\\
For $i = 0,1,\ldots,\frac{n}{4} - 2$, let $A_{i} = n_{(j + i)} = 2{\delta \choose j + i}{n - \delta - 1 \choose 2m - j - i}$,
and as in Case 1, let $B_{i} = n_{(\frac{n}{4}-i)} = 2{\delta \choose m - i}{n - \delta - 1 \choose m + i}$.
As in Case 1, we prove by recursion on $i$ that inequality (4) is satisfied for $A_{i}$ and $B_{i}$ defined here.
Towards this end, let $\delta = \frac{n}{2} + s$ where $s$ is odd. We have that,
\begin{eqnarray}
\frac{A_{0}}{B_{0}}& = &\frac{(\frac{n}{4})!(\delta - \frac{n}{4})!(\frac{n}{4})!(\frac{3n}{4} - \delta - 1)!}
{j!(\delta - j)!(\frac{n}{2}-j)!(\frac{n}{2}-\delta + j - 1)!} \nonumber \\
&=& \frac{(\frac{n}{4})!(\frac{n}{4} + s)!(\frac{n}{4})!(\frac{n}{4} - s -1)!}{(\frac{n}{4} + \frac{s}{2}-\frac{1}{2})!(\frac{n}{4} + \frac{s}{2} +\frac{1}{2})!
(\frac{n}{4} - \frac{s}{2} + \frac{1}{2})!
(\frac{n}{4} - \frac{s}{2} - \frac{3}{2})!} \nonumber \\
&=&\frac{(\frac{n}{4} + s)(\frac{n}{4} + s - 1)\ldots(\frac{n}{4} + \frac{s}{2} + \frac{3}{2})(\frac{n}{4})(\frac{n}{4} - 1)\ldots(\frac{n}{4} - \frac{s}{2} + \frac{3}{2})}{(\frac{n}{4} + \frac{s}{2} - \frac{1}{2})(\frac{n}{4} + \frac{s}{2} -\frac{3}{2})\ldots (\frac{n}{4} + 1)(\frac{n}{4} - \frac{s}{2} - \frac{3}{2})(\frac{n}{4} - \frac{s}{2} - \frac{5}{2})\ldots(\frac{n}{4} - s)}\nonumber \\
& \geq & \frac{(\frac{n}{4} + s)(\frac{n}{4} + s - 1)\ldots(\frac{n}{4} + \frac{s}{2} + \frac{3}{2})(\frac{n}{4} - 1)\ldots(\frac{n}{4} - \frac{s}{2} + \frac{3}{2})}{(\frac{n}{4} + \frac{s}{2} - \frac{1}{2})(\frac{n}{4} + \frac{s}{2} -\frac{3}{2})\ldots (\frac{n}{4} + 1)(\frac{n}{4} - \frac{s}{2} - \frac{3}{2})\ldots(\frac{n}{4} - s + 1)}\frac{\frac{n}{4}}{(\frac{n}{4} - s)} \nonumber
\end{eqnarray}
Now, applications of Lemma 1 give
\begin{eqnarray}
\frac{A_{0}}{B_{0}}& \geq &\frac{{(\frac{n}{4} + \frac{3s}{4} + \frac{3}{4})}^{\left(\frac{s}{2}-\frac{1}{2}\right)}}
{{(\frac{n}{4} + \frac{s}{4} + \frac{1}{4})}^{\left(\frac{s}{2}-\frac{1}{2}\right)}}
\frac{{(\frac{n}{4} - \frac{s}{4} + \frac{1}{4})}^{\left(\frac{s}{2}-\frac{1}{2}\right)}}
{{(\frac{n}{4} - \frac{3s}{4} - \frac{1}{4})}^{\left(\frac{s}{2}-\frac{1}{2}\right)}}\frac{\frac{n}{4}}{(\frac{n}{4} - s)}\nonumber \\
& \geq & \frac{{(\frac{n}{4} + \frac{s}{4} + \frac{1}{2})}^{s-1}}
{{(\frac{n}{4} - \frac{s}{4})}^{s-1}}\frac{\frac{n}{4}}{(\frac{n}{4} - s)}\nonumber \\
& \geq & \frac{{(\frac{n}{4} + \frac{s}{4} + \frac{1}{2})}^{s}}
{{(\frac{n}{4} - \frac{s}{4})}^{s}} \nonumber
\end{eqnarray}
This is exactly inequality (5) obtained in proving Case 1. The rest of the proof for Case 2 is similar to that of Case 1
and we omit it.
\\
\\
\noindent
{\bf Case 3}: $n \equiv 2\pmod{4}$.
\\
In this case we point out that a proof similar to that in cases 1 and 2 above verifies the result.
\\
\\
Remark: We argue that there was no loss of generality in our assumption at the beginning of the proof of Theorem 7 that $d^{+}(v) = d^{-}(v) = \delta(D)$ for each $v \in V(D)$.
Let $D^{*} = (V^{*}, A(D^{*})$ be a directed graph with $d^{+}(v) \geq \delta(D^{*})$, and  $d^{-}(v) \geq \delta(D^{*})$ for each $v \in V(D^{*})$.
 Let $v \in V(D^{*})$, and, let $n_{k}^{*}$ denote the number of equipartitions of $V(D^{*})$ into $V(D^{*}) = X \cup Y$
for which ${\rm deg}(v,B(X,Y)) = k$. We can delete some arcs pointed into $v$ and some arcs pointed out of $v$ to get a directed graph
$D = (V^{*},A(D))$ in which $d^{+}(v) = d^{-}(v) = \delta(D^{*})$. Now as before let $n_{k}$ denote the number of equipartitions of $V(D)$ into $V(D) = X \cup Y$
for which ${\rm deg}(v,B(X,Y)) = k$. It is clear that $\sum_{k = 2}^{q}n_{k} \geq \sum_{k = 2}^{q}n_{k}^{*}$ for each $q$,
and that $ \sum_{k = \delta - \frac{n}{2} +1}^{\frac{n}{2}} n_{k} = \sum_{k = \delta - \frac{n}{2} +1}^{\frac{n}{2}} {n_{k}}^{*}$ = total
number of equipartitions of $V(D^{*})$. Hence, the proof above that $T > S$ holds with $n_{k}$ replaced by $n_{k}^{*}$.
\end{mainproof}
\\
\\
We now prove the corollaries of Theorem 7 mentioned in the introduction.
\\
\\
\noindent
\begin{cor1proof}
 If $n \leq 10$ then $\delta(D) > \frac{2}{3}n$ and Theorem 6 implies that
$D$ has an anti-directed Hamilton cycle. Hence, assume that $n > 10$, and for given $n$,
let $p$ be the unique real number such that
$\frac{1}{2} < p < \frac{3}{4}$ and $n = \frac{{\rm ln}(4)}{\left(p - \frac{1}{2}\right){\rm ln}\left(\frac{p + \frac{1}{2}}{\frac{3}{2} - p}\right)}$. The result follows from Theorem 7 if $\delta(D) > pn$ and since  $\delta(D) > \frac{1}{2}n  + \sqrt{n\ln(2)}$, it suffices to show that
$pn \leq \frac{1}{2}n  + \sqrt{n\ln(2)}$. Let $x = p - \frac{1}{2}$ and note that $0 < x < \frac{1}{4}$.
Now, $pn \leq \frac{1}{2}n  + \sqrt{n\ln(2)}$ if and only if $xn \leq \sqrt{n\ln(2)}$ if and only if
$\sqrt{\frac{\ln(4)}{x\ln\left(\frac{1 + x}{1 - x} \right)}} \leq \frac{\sqrt{\ln(2)}}{x}$ if and only if $2x \leq \ln(1+x) - \ln(1-x)$.
 Since $0 < x < \frac{1}{4}$,
we have that $\ln(1+x) - \ln(1-x) = \sum_{k = 0}^{\infty}\frac{2x^{2k + 1}}{2k + 1}$ and this completes the proof of Corollary 1.
\end{cor1proof}
\\
\\
\noindent
\begin{cor2proof} For $p = \frac{9}{16}$, $177 < \frac{{\rm ln}(4)}{\left(p - \frac{1}{2}\right){\rm ln}\left(\frac{p + \frac{1}{2}}{\frac{3}{2} - p}\right)} < 178$. Hence, Theorem 7 implies that the corollary is true for all $n \geq 178$. If $n < 178$, $\delta(D) > \frac{9}{16}n$, and,
$n \not\equiv 0\pmod{4}$, we can
verify that inequality (3) is satisfied by direct computation. If $n < 178$, $\delta(D) > \frac{9}{16}n$, and,
$n \equiv 0\pmod{4}$, a use of Theorem 8 that is stronger than its use in deriving the bound $S$ in equation (2) yields that
the number of equipartitions of $V(D)$ into $V(D) = X \cup Y$ for which
$B(X,Y)$ does not contain a Hamilton cycle is at most
\begin{equation}
S' = n\left(\frac{n_{2}}{2} + \frac{n_{3}}{3} + \ldots + \frac{n_{\lfloor \frac{n}{4}\rfloor}}{2\lfloor\frac{n}{4}\rfloor}\right).
\end{equation}
Direct computation now verifies that $T > S'$.
\end{cor2proof}
\\
\\
\noindent
\begin{cor3proof} If $n \leq 14$ is even and $\delta(D) > \frac{1}{2}n$ then we have that $\delta(D) > \frac{9}{16}n$
and Corollary 2 implies Corollary 3.
\end{cor3proof}

\end{document}